\documentclass[12pt]{article}
\usepackage[english,activeacute]{babel}
\usepackage{amsmath,amsfonts,amssymb,amstext,amsthm,amscd,mathrsfs,amsbsy}

\newtheorem{teo}{Theorem}[section]
\newtheorem{pro}[teo]{Proposition}
\newtheorem{coro}[teo]{Corollary}
\newtheorem{lem}[teo]{Lemma}

\theoremstyle{definition}
\newtheorem{defi}[teo]{Definition}
\newtheorem{exam}[teo]{Example}
\newtheorem{rem}[teo]{Remark}

\newcommand{\N}{\mathbb N}
\newcommand{\ZZ}{\mathbb{Z}}
\newcommand{\V}{\mathbf V}
\newcommand{\K}{\mathbb K}
\newcommand{\Po}{\K[x_1,\ldots,x_r]}
\newcommand{\mm}{\mathfrak{m}}
\newcommand{\lb}{\lambda}
\newcommand{\gr}{X_A}
\newcommand{\Sn}{S^{(n)}}
\newcommand{\Snu}{S^{(n+1)}}

\newcommand{\lag}{\langle}
\newcommand{\rag}{\rangle}
\DeclareMathOperator{\ri}{ri_{\underline{0}}}

\begin{document}

\title{On the multiplicity and regularity index of toric curves}

\author{Daniel Duarte\footnote{Research supported by CONACyT grant 287622.}, Alondra Ram\'irez-Sandoval}

\maketitle

\begin{abstract}
In this note we revisit the problem of determining combinatorially the multiplicity at the origin of a toric curve.
In addition, we give the exact value of the regularity index of that point for plane toric curves and effective 
bounds for this number for arbitrary toric curves.
\end{abstract}


\section*{Introduction}

A classical numerical invariant associated to a point of an algebraic variety is the (Hilbert-Samuel) multiplicity. This invariant has 
numerous applications on algebraic geometry and commutative algebra (for instance, it plays a fundamental role in the problem of 
resolution of singularities). On the other hand, it is known that this invariant can be computed using the theory of Gr\"obner bases 
(see \cite[Chapter 5]{GP}).

In the context of toric varieties the computation of this invariant is simpler. Assuming that an affine toric variety contains the origin, 
its multiplicity can be computed in a combinatorial way by looking at the semigroup defining the toric variety (see 
\cite[Chapter 5, Theorem 3.14]{GKZ}). The proof of this fact, although not too difficult, requires some tools from topology, 
commutative algebra, and combinatorics. The first goal of this note is to give a completely elementary proof of that fact in the case
of toric curves.

Let $S\subset\N$ be a numerical semigroup generated by ${a_1,\ldots,a_r}$, where $a_1<\cdots<a_r$. The corresponding toric 
curve contains the origin and its multiplicity can be easily computed: it is $a_1$. We will prove this fact using only basic concepts
of numerical semigroups. In addition, the proof we present is sufficiently constructive so that a related question can be answered 
along the way.

Recall that for a point in a curve, say $p\in X\subset\K^r$, its multiplicity is defined as the number $m\in\N$ such that 
$\dim_{\K}\mm^{n}/\mm^{n+1}=m$ for all $n\gg0$, where $\mm\subset\K[X]$ denotes the maximal ideal corresponding to $p$.
It is natural to ask for the minimum $k\in\N$ for which the previous statement holds. Some related questions have been studied 
by several authors (see, for instance, \cite{M} or \cite{St2}). The second goal of this note is to describe exactly this minimum for 
plane toric curves and to give effective bounds for this number for arbitrary toric curves.
 

\section{Numerical semigroups and toric curves}

In this section we recall the basic definitions and facts regarding numerical semigroups and toric curves. The notation introduced in 
this section will be used throughout this note.

Let $A=\{a_1,a_2,\ldots,a_r\}\subset\N$ be such that $1<a_1<a_2<\cdots<a_r$ and $\gcd(A)=1$. We denote as 
$S:=\lag A \rag$ the numerical semigroup generated by $A$, and we assume that $A$ is the minimal generating set of $S$
(see \cite{RG} for generalities on numerical semigroups). 

Given a numerical semigroup, we can associate a toric curve. We are interested in the study of some relations among numerical 
invariants of toric curves and data of the corresponding semigroup. Let us recall the general definition of an affine toric variety 
(see, for instance, \cite[Section 1.1]{CLS} or \cite[Chapter 4]{St1}).

Let $A=\{a_1,\ldots,a_r\}\subset\ZZ^d$ be such that $\ZZ A=\{\sum_i \lambda_i a_i|\lambda_i\in\ZZ\}=\ZZ^d$. 
The set $A$ induces a homomorphism of semigroups
\begin{align}\label{e. pi}
\pi_{A}:\N^r\rightarrow\ZZ^d,\mbox{ }\mbox{ }\mbox{ }\alpha=(\alpha_1,\ldots,\alpha_r)
\mapsto \alpha_1 a_1+\cdots+\alpha_r a_r.\notag
\end{align}
Let $\K$ be a field and consider the ideal (by using the usual multi-index notation):
$$I_{A}:=\langle x^{\alpha}-x^{\beta}|\alpha,\beta\in\N^r,\mbox{ }\pi_{A}(\alpha)=\pi_{A}
(\beta)\rangle\subset\Po.$$

\begin{defi}
We call $X_A:=\V(I_A)\subset\K^r$ the toric variety defined by $A$.
\end{defi}

It is well known that a variety obtained in this way is irreducible, contains a dense open set isomorphic 
to $(\K^*)^d$ and the natural action of $(\K^*)^d$ on itself extends to an action on $X_A$. 

Now we introduce a numerical invariant associated to a toric curve. Let $A=\{a_1,\ldots,a_r\}\subset\N$ be as before. 
Let $\mm=\lag t^{a_1},\ldots,t^{a_r} \rag\subset\K[t^{a_1},\ldots,t^{a_r}]=\K[X_A]$ be the maximal ideal corresponding to 
$\underline{0}:=(0,\ldots,0)\in X_A$. For $n\in\N$, we denote 
$$\lb(n):=\dim_{\K}\mm^{n}/\mm^{n+1}.$$
A classical result states that $\lb(n)=m$ for some $m\in\N$, for all $n\gg0$. This number is called the \textit{multiplicity of} 
$X_A$ \textit{in} $\underline{0}$. Another well known result states that $m$ is actually $a_1$ (this is proved in greater generality 
in \cite{L} or \cite[Chapter 5, Theorem 3.14]{GKZ}). In addition, we denote
$$\ri(X_A):=\min\{k\in\N|\lb(n)=a_1\mbox{ for all }n\geq k\}.$$
The number $\ri(X_A)$ is called the \textit{regularity index of the origin of the toric curve} $X_A$. In what follows we give an
elementary proof of the fact $\lb(n)=a_1$ for all $n\gg0$, as well as effective bounds for $\ri(X_A)$.


\section{Multiplicity and regularity index of plane toric curves}

In this section we study the special case of plane toric curves. This case is particularly simple and the regularity index can
be computed explicitly.

\begin{pro}
Let $A=\{a_1,a_2\}$ be as before. Then $\lb(n)<a_1$ if $n<a_1-1$ and $\lb(n)=a_1$ if $n\geq a_1-1$. 
In particular, $\ri(\gr)=a_1-1$.
\end{pro}
\begin{proof}
Let $\mm=\lag t^{a_1},t^{a_2}\rag\subset\K[X_A]$. For $n\geq1$, a $\K$-basis for $\mm^n$ is the set
$\{t^{(a_1, a_2) \cdot \alpha }\mid  \alpha\in \N^2, |\alpha|\geq n\}$, where $(a_1, a_2) \cdot \alpha$ denotes
the standard dot product. Thus, a generating set for $\mm^n/\mm^{n+1}$, as $\K$-vector space, is given by 
$\{ t^{(a_1, a_2)\cdot \alpha}+\mm^{n+1}\mid\alpha\in\N^2,|\alpha|=n\}.$ Let $c$ be 
the cardinality of this set. If $n<a_1-1$ then $\lb(n)\leq c\leq n+1<a_1-1+1=a_1$. This shows the first statement of the proposition.

Now assume that $n\geq a_1-1$. We claim that 
$$B_n:=\{t^{(a_1,a_2)\cdot(n-j,j)}+ \mm^{n+1} \mid j\in \{0,\ldots,a_1-1\} \}$$
is a $\K$-basis for $\mm^n/\mm^{n+1}$. To see that it is a generating set it is enough to show that 
$t^{(a_1,a_2)\cdot(n-j,j)}\in\mm^{n+1}$ if $j\geq a_1$. Let $j=qa_1+r$, where $1\leq q$ and $0\leq r<a_1$. Then,
since $n+(a_2-a_1)q\geq n+1$ we conclude
$$t^{a_1(n-j)+a_2j}=t^{a_1(n-qa_1-r)+a_2(qa_1+r)}=t^{a_1(n+(a_2-a_1)q-r)+a_2r}\in\mm^{n+1}.$$ 

Now suppose that $\sum_{j=0}^{a_1-1}\lb_jt^{(a_1,a_2)\cdot(n-j,j)}=\sum_{|\beta|\geq n+1}b_{\beta}t^{(a_1,a_2)\cdot\beta}$, 
for some $\lambda_i,b_{\beta}\in\K$. Let us define $B'_n:= \{a_1(n-j)+a_2j|j\in\{0,\ldots,a_1-1\}\}.$ We claim that
$\{(a_1,a_2) \cdot \beta \mid |\beta| \geq n+1\} \cap B'_n= \emptyset.$ This implies that $\lb_i=b_{\beta}=0$ for all $i$ and all $\beta$,
i.e., $B_n$ is linearly independent.

Let us prove the claim.
Suppose that $a_1\beta_1+a_2\beta_2=a_1(n-j)+a_2j,$ where $n+1 \leq \beta_1+\beta_2,$ $0\leq \beta_1,$ $0\leq \beta_2$ and 
$0 \leq j\leq a_1-1.$ Then $a_1(\beta_1-n+j)=a_2(j-\beta_2).$ In particular, $a_1 \mid j-\beta_2$. If $\beta_1-n+j=0,$ then 
$j- \beta_2=0$ implying $\beta_1+ \beta_2=n$, which is a contradiction. Suppose $\beta_1-n+j >0$. Then $j-\beta_2 >0.$ 
Since $a_1 \mid j-\beta_2,$ we have $j-\beta_2 \geq a_1$. Thus $\beta_2\leq j-a_1\leq -1$, a contradiction. Finally, suppose 
$\beta_1-n+j <0$. Then $0 \leq \beta_1 <n-j$. Let $\beta_1=n-j-l$ for some $l\in\{ 1, \dots, n-j\}$. Then 
$\beta_2 \geq n+1-\beta_1=j+l+1$ and so we obtain the following contradiction:
$$a_1(n-j)+a_2j<a_1(n-j-l)+a_2(j+l)<a_1(n-j-l)+a_2\beta_2=a_1\beta_1+a_2\beta_2.$$
These contradictions prove the claim. Finally, the same claim also implies that $\lb(n)=|B_n|=a_1$.
\end{proof}


\section{Multiplicity and bounds for the regularity index of toric curves}

In this section we show that the multiplicity of the origin of a toric curve coincides with the smallest non-zero element of the corresponding
semigroup. In addition, we give a bound for its regularity index. The bound is given in terms of the so-called Frobenius number, which we
now define.

Let $A=\{a_1,\ldots,a_r\}\subset\N$ be as before, and let $S=\lag A \rag$. It is well known that any sufficiently large integer belongs 
to $S$. The Frobenius number of $S$, denoted as $F(S)$, is the largest integer that does not belong to $S$.

Let $a=(a_1, \dots, a_r)\in\N^r$ and $\alpha=(\alpha_1, \dots, \alpha_r ) \in \N^r.$ As before, $a\cdot\alpha$ denotes the usual 
dot product. Let $n\in\N$, $n\geq1$, and consider the following sets:
$$\Sn:= \{0\}\cup\{a\cdot\alpha\mid\alpha\in\N^r, |\alpha|\geq n\}.$$
Notice that $\Sn$ is a numerical semigroup and $S=S^{(1)} \supset S^{(2)} \supset \cdots$.

In all that follows we assume $a_1<a_2<F(S)$. A brief discussion of the cases $F(S)<a_1$ and $a_1<F(S)<a_2$ will be given
at the end of this section.

\begin{lem}\label{division}
Let $S=\lag a_1,\ldots,a_r\rag$ be as before. Let $\delta=F(S)-a_1$, $\epsilon=a_2-a_1$, and $\delta=q\epsilon+\tau$, where
$1\leq q$ and $0\leq\tau<\epsilon.$ Then, for every $n\geq q$:
$$F(\Snu)=F(\Sn)+a_1<(n+1)a_2.$$
\end{lem}
\begin{proof}
First notice that $F(\Sn)+a_1+l\in\Snu$ for all $l\geq1$ and so $F(\Snu)\leq F(\Sn)+a_1$ for all $n\geq1$. Thus, for any given $n$,
$$F(\Snu)\leq F(\Sn)+a_1\leq F(S^{(n-1)})+2a_1\leq\cdots\leq F(S)+na_1.$$
Let $n\geq q$. From the equation $\delta=q\epsilon+\tau$ we obtain 
\begin{align}
F(S)+na_1&=q(a_2-a_1)+\tau+(n+1)a_1\notag\\
&<q(a_2-a_1)+a_2-a_1+(n+1)a_1\notag\\
&=(q+1)a_2+(n-q)a_1\notag\\
&<(q+1)a_2+(n-q)a_2=(n+1)a_2.\notag
\end{align}
We conclude that $F(\Snu)\leq F(\Sn)+a_1<(n+1)a_2$ for all $n\geq q$. It remains to prove that $F(\Snu)=F(\Sn)+a_1$ for $n\geq q$.
Suppose that $F(\Sn)+a_1=a\cdot\alpha\in\Snu$, i.e., $|\alpha|=n+1$. If $\alpha_1\geq1$ then $F(\Sn)=a\cdot\alpha-a_1\in\Sn$, 
which is a contradiction. If $\alpha_1=0$ then $F(\Sn)+a_1=\sum_{i=2}^ra_i\alpha_i\geq(n+1)a_2$, a contradiction.
Therefore, $F(\Sn)+a_1\notin\Snu$ and so $F(\Snu)=F(\Sn)+a_1$ for all $n\geq q$.
\end{proof}

\begin{pro}\label{S(n+1)=S(n)+a_1}
Let $S= \lag a_1, \dots, a_r  \rag$. Then $S^{(n+1)}=S^{(n)}+a_1$ for every $n\geq q$.
\end{pro}
\begin{proof}
By definition, $\Sn+a_1\subset\Snu$, for every $n\geq1$. Assume that $n\geq q$ and let $s \in S^{(n+1)}$. Suppose $s>F(\Snu)$.
By lemma \ref{division}, $F(\Snu)=F(\Sn)+a_1$ and so, for some $l\geq1$, $s=F(\Sn)+a_1+l=(F(\Sn)+l)+a_1\in\Sn+a_1$.

Now suppose that $s<F(\Snu)$. Let $s=a\cdot\alpha$, where $|\alpha|\geq n+1$. If 
$\alpha_1\geq1$ then $s=(a\cdot\alpha-a_1)+a_1\in\Sn+a_1$. If $\alpha_1=0$ then $s=\sum_{i=2}^{r}a_i\alpha_i\geq(n+1)a_2$.
On the other hand, by lemma \ref{division}, $s<F(\Snu)<(n+1)a_2$, which is a contradiction. We conclude that $\Snu\subset\Sn+a_1$.
\end{proof}

Now we are ready to prove our main theorem.

\begin{teo}\label{main}
Let $A=\{a_1,\ldots, a_r\}$ be as before and $S=\lag A \rag$. Let $\delta=F(S)-a_1$, $\epsilon=a_2-a_1$, and $\delta=q\epsilon+\tau$, 
where $1\leq q$ and $0\leq\tau<\epsilon.$ Then $\lb(n)=a_1$ for every $n\geq q$. In particular, $\ri(\gr)\leq q$.
\end{teo}
\begin{proof}
Let $\mm= \lag t^{a_1},\ldots,t^{a_r}\rag\subset \K[X_A]$ be the maximal ideal corresponding to $\underline{0}\in X_A$. 
A $\K$-basis for $\mm^n$ is given by $\{t^\gamma\mid\gamma\in S^{(n)}\}.$ In particular, a $\K$-basis for $\mm^n/ \mm^{n+1}$ 
is given by $\{t^{\gamma}+\mm^{n+1} \mid \gamma \in S^{(n)}\backslash S^{(n+1)} \}.$ We claim that 
$|S^{(n)}\backslash S^{(n+1)}|=a_1$ for every $n\geq q$.

Consider the following sets:
\begin{align}
A&:= \{ s \in S^{(n)} \mid s<(n+1)a_1 \},\notag\\
B&:= \{ h+ka_1 \mid h<(n+1)a_1, h \notin S^{(n)}, \medspace \text{and} \medspace k:= \min \{ l \mid h+la_1 \in S^{(n)}\} \}.\notag
\end{align}
Let us prove that $S^{(n)} \backslash  S^{(n+1)}= A \cup B$, for $n\geq q$. This proves the theorem since $|A\cup B|=a_1$.

Let $s\in A \cup B$. If $s \in A$, then $s\in S^{(n)}$ and $s<(n+1)a_1$. Thus, $s\in S^{(n)} \backslash  S^{(n+1)}$.
If $s\in B,$ then $s=h+ka_1$, where $h<(n+1)a_1, h \notin S^{(n)}$, and $k=\min \{ l \mid h+ la_1 \in S^{(n)}\}$. 
Suppose that $s \in S^{(n+1)}$. By proposition \ref{S(n+1)=S(n)+a_1} we have that $h+ka_1=t+a_1$, for some $t \in S^{(n)}$. 
Then $h+(k-1)a_1=t$, which contradicts the minimality of $k.$ Therefore, $s\in S^{(n)} \backslash  S^{(n+1)}.$

Now let $s \in S^{(n)} \backslash  S^{(n+1)}$. If $s<(n+1)a_1$ then $s\in A$. If $s>(n+1)a_1$, let $s=pa_1+t$, where 
$0\leq t < a_1$ and $n+1\leq p$. Notice that $t=0$ implies $s=pa_1\in S^{(p)} \subset S^{(n+1)}$, a contradiction. Thus, 
$0<t<a_1.$ We claim that $t+na_1\notin S^{(n)}.$  Suppose that $t+na_1=a\cdot\alpha\in S^{(n)},$ i.e., $|\alpha|\geq n$. Then 
$s+na_1=pa_1+t+na_1=pa_1+a\cdot\alpha$. Since $p-n \geq 1,$ we have $s=a\cdot\alpha+(p-n)a_1\in S^{(n+1)}$, 
contradicting again that $s \in S^{(n)} \backslash  S^{(n+1)}.$ This proves the claim.  

Let $h:= t+na_1$ and $k:=p-n.$ Then $s=h+ka_1.$ Suppose that $j<k$ is such that $h+ja_1 \in S^{(n)},$ i.e., 
$h+ja_1=a\cdot\beta$, where $|\beta|\geq n.$ Let $k=j+l$ for some $l\geq 1.$ Then 
$s= h+(j+l)a_1  =h+ja_1+la_1 =a\cdot\beta+la_1\in S^{(n+1)}$, a contradiction. Therefore $k$ is the minimal element such that
$h+ka_1 \in S^{(n)}$ and so $s\in B$.
\end{proof}

In the following example we show that the bound for the regularity index given in theorem \ref{main} is sharp.
\begin{exam}
Let $a_1\geq3$ and $S=\lag a_1,2a_1+1,2a_1+2,\ldots,2a_1+(a_1-2)\rag$. In this example, $F(S)=4a_1-1$ and $q=2$. 
Since the generating set of $S$ has cardinality $a_1-1$, we have $\lb(1)=a_1-1$. Therefore, $\ri(\gr)=2=q$.
\end{exam}

The bound given in theorem \ref{main} is given in terms of the Frobenius number of the semigroup. Unfortunately, it is well-known
that this number cannot be explicitly computed in general. However, there are some known bounds for it.

\begin{coro}
With the notation of the theorem, 
$$\ri(\gr)\leq \frac{1}{a_2-a_1}\Big(a_1a_r-2a_1-a_r-\tau\Big).$$
\end{coro}
\begin{proof}
It follows from the theorem and the fact $F(S)\leq a_1a_r-a_1-a_r$ (see \cite{B}).
\end{proof}

Throughout this section, we considered the case $a_1<a_2<F(S)$. Let us conclude with a discussion on the two other special cases:
$F(S)<a_1$ and $a_1<F(S)<a_2$.

\begin{rem}
Suppose that $F(S)=a_1-1$, i.e., $S=\{0,a_1,a_1+1,\ldots\}$. We already checked that $\Sn+a_1\subset\Snu$ for all $n\geq1$. 
In addition, since $\min\Snu\setminus\{0\}=(n+1)a_1$, it follows that $\Snu=\Sn+a_1$ for all $n\geq1$. Following the proof of 
theorem \ref{main} we obtain that $\lb(n)=a_1$ for all $n\geq1$, i.e., $\ri(\gr)=1$.
\end{rem}

\begin{rem}
Suppose that $a_1<F(S)<a_2$. First notice that $F(\Sn)\leq F(S^{(n-1)})+a_1\leq\cdots\leq F(S)+(n-1)a_1<a_2+(n-1)a_1<na_2$, 
for every $n\geq1$. In addition, the last paragraph of the proof of lemma \ref{division} also applies to show that $F(\Snu)=F(\Sn)+a_1$,
for all $n\geq1$. These facts were the main ingredients to prove that $\Snu=\Sn+a_1$. As in the previous remark, we conclude 
that $\lb(n)=a_1$ for all $n\geq1$, i.e., $\ri(\gr)=1$.
\end{rem}

\section*{Acknowledgements}

We want to thank Enrique Ch\'avez for suggesting us an idea to conclude the proof of theorem \ref{main}. We also thank
Mario Huicochea for explaining us some results on sumsets and for the reference \cite{B}. Finally, we thank Omar Antol\'in for
stimulating discussions on the topic of this note.

\vspace{.5cm}
{\footnotesize \textsc {D. Duarte, Universidad Aut\'onoma de Zacatecas-CONACYT.} \\
E-mail: aduarte@uaz.edu.mx}\\
{\footnotesize \textsc {A. Ram\'irez-Sandoval, Universidad Aut\'onoma de Zacatecas.} \\
E-mail: acrmz96@gmail.com}


\begin{thebibliography}{XXX}
\bibitem[B]{B}Brauer, A.; \textit{On a problem of partitions}, Amer. J. Math. \textbf{64}, pp. 299-312, (1942).
\bibitem[CLS]{CLS}Cox, D., Little, J., Schenck, H.; \textit{Toric Varieties}, Graduate Studies in Mathematics, Volume 124,
AMS, 2011.
\bibitem[GKZ]{GKZ} Gelfand I. M., Kapranov M. M., Zelevinsky A. V.; \textit{Discriminants, Resultants and Multidimensional Determinants}, 
Birkh\"auser, USA, 1994.
\bibitem[GP]{GP} Greuel, G.-M., Pfister, G.; \textit{A Singular Introduction to Commutative Algebra}, Springer, 2nd. Ed., 2008.
\bibitem[L]{L} Lipman J.; \textit{Stable ideals and Arf rings},  Amer. J. Math. \textbf{93}, pp. 649–685, (1971).
\bibitem[M]{M} Morales, M.; \textit{Syzygies of monomial curves and a linear diophantine problem of Frobenius},  Internal Report,
Max Planck Institut f\"ur Mathematik, Bonn, (1987).
\bibitem[RG]{RG}Rosales J. C., Garc\'ia-S\'anchez, P. A.; \textit{Numerical Semigroups}, \textbf{20}, Springer, 2009.
\bibitem[St1]{St1}Sturmfels, B.; \textit{Gr\"obner Bases and Convex Polytopes}, University Lecture Series, Vol. 8,
American Mathematical Society, Providence, RI, 1996.
\bibitem[St2]{St2}Sturmfels, B.; \textit{Equations defining toric varieties}, Algebraic geometry - Santa Cruz 1995,  
Proc. Sympos. Pure Math., \textbf{62}, Part 2, Amer. Math. Soc., Providence, RI, pp. 437-449, (1997).
\end{thebibliography}
\end{document}